\documentclass[10pt]{article}
\title{An extension of the proximal point algorithm beyond convexity}
\author{S.-M. Grad,\thanks{Faculty of Mathematics, University of Vienna, Oskar-Morgenstern-Platz 1, A-1090 Vienna, Austria. ~ E-mail: sorin-mihai.grad@univie.ac.at, ORCID-ID: 0000-0002-1139-7504, Corresponding author} ~~~
F. Lara\thanks{Departamento de Matem\'atica, Facultad de Ciencias, Universidad de Tarapac\'a, Arica, Chile. ~ E-mail: felipelaraobreque@gmail.com; flarao@uta.cl}}
\date{\small \emph{\today}}
\usepackage{amsmath}
\usepackage{amsthm}
\usepackage{amssymb}
\usepackage{multicol}
\usepackage[active]{srcltx}
\usepackage{algorithm}
\usepackage{array}
\usepackage[dvipsnames,svgnames]{xcolor}

\numberwithin{equation}{section}

\DeclareMathOperator*\prox{Prox}%
\DeclareMathOperator*\dom{dom}%
\DeclareMathOperator*\epi{epi}%
\DeclareMathOperator*\bd{bd}%
\DeclareMathOperator*\inte{int}%
\DeclareMathOperator*\id{Id}%
\DeclareMathOperator*\amin{\arg\min}%
\newtheorem{lemma}{Lemma}[section]
\newtheorem{theorem}{Theorem}[section]
\newtheorem{corollary}{Corollary}[section]
\newtheorem{proposition}{Proposition}[section]
\newtheorem{example}{Example}[section]
\newtheorem{definition}{Definition}[section]
\newtheorem{remark}{Remark}[section]

\begin{document}
\maketitle
\begin{abstract}

\noindent {\bf Abstract.} {\small We introduce and investigate a new genera\-li\-zed convexity notion for functions called prox-convexity. The proximity operator of such a function is single-valued and firmly nonexpansive. We provide examples of (strongly) quasiconvex, weakly convex, and DC (difference of convex) functions that are prox-convex, however none of these classes fully contains the one of prox-convex functions or is included into it. We show that the classical proximal point algorithm remains convergent when the convexity of the proper lower semicontinuous function to be minimized is relaxed to prox-convexity.} 

\medskip

\noindent{\small \emph{Keywords.} Nonsmooth optimization, Nonconvex optimization, Proximity operator, Proximal point algorithm, Generalized convex function}
\end{abstract}
\maketitle

\section{Introduction}\label{sec:1}

The first motivation behind this study comes from works like \cite{CCO,LAT,PCH,PMO}
where proximal point type methods for minimizing quasiconvex functions formulated by means of Bregman distances were proposed. On the other hand, other
extensions of the proximal point algorithm for nonconvex optimization pro\-blems (such as the ones introduced in \cite{COP, IPS, LEW, PEN}) cannot be employed in such situations due to various reasons. Looking for a way to reconcile these approaches we came across a new class of generalized convex functions that we called prox-convex, whose properties allowed us to extend the convergence of the classical proximal point algorithm beyond the convexity setting into a yet unexplored direction.

In contrast to other similar generalizations, the proximity operators of the proper prox-convex functions are single-valued (and firmly nonexpansive) on the underlying sets. To the best of our knowledge besides the convex and prox-convex functions only the weakly convex ones have single-valued proximity operators (cf. \cite{HSA}). This property plays a crucial role in the construction of proximal point type algorithms as the new iterate is thus uniquely determined and does not have to be picked from a set. Moreover, the prox-convexity of the functions can be considered both globally or on a subset of their domains, that can be of advantage when dealing with concrete applications from practice. Various functions, among which several families of (strongly) quasiconvex, weakly and DC (i.e. difference of convex) ones, fulfill the definition of the new notion we propose. As a byproduct of our study we also deliver new results involving (strongly) quasiconvex functions.

Different to other extensions of the proximal point algorithm, the one we propose has a sort of a local nature, however not in the sense of properties of a function that hold in some neighborhoods, but concerning the restriction of the function to a (convex) set. We are not aware of very similar work in the literature where the proximity operator of a function is taken with respect to a given set, however in works like \cite{BCF, GRN} such constructions with some employed functions not split from the corresponding sets
were already considered.

Given a proper, lower semicontinuous and convex function $h: \mathbb{R}^{n} \rightarrow \overline{\mathbb{R}} := \mathbb{R} \cup \{\pm \infty\}$, for any $z \in \mathbb{R}^{n}$ the minimization problem
\begin{equation}\label{main:prox}
 \min_{x \in \mathbb{R}^{n}} \left[h(x) + \frac{1}{2} \lVert z - x \rVert^{2}\right]
\end{equation}
has (even in more general frameworks such as Hilbert spaces) a unique optimal solution denoted by $\prox_{h} (z)$, that is the value of the \textit{proximity operator} of the function $h$ at the point $z$.
A fundamental pro\-per\-ty of the latter is when $z, \overline{x} \in \mathbb{R}^{n}$ (see, for instance, \cite[Proposition 12.26]{BAC})
\begin{equation}\label{fundamental}
 \overline{x} = {\prox}_{h} (z) ~ \Longleftrightarrow ~ z - \overline{x} \in \partial h(\overline{x}),
\end{equation}
where $\partial h$ is the usual convex subdifferential.

These two facts (the existence of an optimal solution to \eqref{main:prox} and the
cha\-rac\-terization \eqref{fundamental}) are crucial tools for proving the
convergence of the proximal point type algorithms for continuous optimization
problems consisting in minimizing (sums of) proper, lower semicontinuous and convex
functions, and even for DC programming problems (see \cite{BBO} for instance).
For the class of prox-convex functions introduced in this article the first of them holds while the second one is replaced by a weaker variant and we show that these properties still guarantee the convergence of the sequence generated by the proximal point algorithm towards a minimum of a prox-convex function.

The paper is constructed as follows. After some preliminaries, where we de\-fi\-ne
the framework and recall some necessary notions and results, we introduce and
investigate the new classes of prox-convex functions and strongly G-subdifferentiable
functions, showing that the proper and lower semicontinuous elements of the latter
belong to the first one, too. Finally, we show that the classical proximal point
algorithm can be extended to the prox-convex setting without losing the convergence.

\section{Preliminaries}

By $\langle \cdot,\cdot \rangle$ we mean the \textit{inner product} of $\mathbb{R}^{n}$ and by $\lVert \cdot \rVert$ the Euclidean \textit{norm} on $\mathbb{R}^{n}$. Let $K$ be a nonempty set in $\mathbb{R}^{n}$ and we denote its \textit{topological interior} by
$\inte K$ and its \textit{boundary} by $\bd K$. The \textit{indicator function} of $K$ is defined by $\delta_{K} (x) := 0$ if $x \in K$, and $\delta_{K} (x):= + \infty$ elsewhere. By $\mathbb{B} (x, \delta)$ we mean the \textit{closed ball} with center at $x\in \mathbb{R}^{n}$ and radius $\delta > 0$. By $\id:\mathbb{R}^n\to \mathbb{R}^n$ we denote the \textit{identity mapping} on $\mathbb{R}^n$.

Given any $x, y, z \in \mathbb{R}^{n}$, we have
 \begin{equation}\label{3:points}
  \langle x - z, y - x \rangle = \frac{1}{2} \lVert z - y \rVert^{2} - \frac{1}{2}
  \lVert x - z \rVert^{2} - \frac{1}{2} \lVert y - x \rVert^{2}.
 \end{equation}

For any $x, y \in \mathbb{R}^{n}$ and any $\beta \in \mathbb{R}$, we have
 \begin{equation}
  \lVert \beta x + (1-\beta) y \rVert^{2} = \beta \lVert x \rVert^{2} + (1 -
  \beta) \lVert y\rVert^{2} - \beta (1 - \beta) \lVert x - y \rVert^{2}.
 \end{equation}

Given any extended-valued function $h:\mathbb{R}^{n}\rightarrow \overline
{\mathbb{R}} := \mathbb{R}\cup \{ \pm \infty \}$, the \textit{effective domain}
of $h$ is defined by $\dom \,h := \{x \in \mathbb{R}^{n}: h(x) < + \infty \}$.
We say that $h$ is \textit{proper} if $\dom \,h$ is nonempty and $h(x) > -
\infty$ for all $x \in \mathbb{R}^{n}$.

We denote by $\epi h := \{(x,t) \in \mathbb{R}^{n} \times \mathbb{R}: h(x)
\leq t\}$ the \textit{epigraph} of $h$, by $S_{\lambda} (h) := \{x \in \mathbb{R}^{n}:
h(x) \leq \lambda\}$ (respectively $S^{<}_{\lambda} (h) := \{x \in \mathbb{R}^{n}: h(x)
< \lambda\}$ the \textit{sublevel} (respectively \textit{strict sublevel}) \textit{set} of $h$ \textit{at the height}
$\lambda \in \mathbb{R}$, and by ${\amin}_{\mathbb{R}^{n}} h$ the set of all minimal points of $h$.
We say that a function is \textit{$L$-Lipschitz} when it is Lipschitz continuous
with constant $L > 0$. We adopt the usual convention $\sup_{\emptyset} h := -
\infty$ and $\inf_{\emptyset} h := + \infty$.

A function $h$ with a convex domain is said to be
\begin{itemize}
 \item[$(a)$] \textit{convex} if, given any $x, y \in \dom \,h$, then
 \begin{equation}
  h(\lambda x + (1-\lambda)y) \leq \lambda h(x) + (1 - \lambda) h(y), ~ \forall ~
 \lambda \in [0, 1];
 \end{equation}
 \item[$(b)$] \textit{semistrictly quasiconvex} if, given any $x, y \in \dom ~h,$
 with $h(x) \neq h(y)$, then
 \begin{equation}
  h(\lambda x + (1-\lambda)y) < \max \{h(x), h(y)\}, ~ \forall ~ \lambda \in
 \, ]0, 1[;
 \end{equation}
 \item[$(c)$] \textit{quasiconvex} if, given any $x, y \in \dom \,h$, then
 \begin{equation}\label{def:qcx}
  h(\lambda x + (1-\lambda) y) \leq \max \{h(x), h(y)\}, ~ \forall ~ \lambda
 \in [0,1].
 \end{equation}
 \noindent We say that $h$ is \textit{strictly quasiconvex} if the inequality in
 \eqref{def:qcx} is strict (see \cite[page 90]{HKS}).
\end{itemize}
Every convex function is quasiconvex and semistrictly quasiconvex,
and every semistrictly quasiconvex and lower semicontinuous function is quasiconvex (see \cite[Theorem 2.3.2]{CMA}). The function $h:
\mathbb{R} \rightarrow \mathbb{R}$, with $h(x) := \min\{\lvert x \rvert, 1\}$,
is quasiconvex, without being semistrictly quasiconvex.

A function $h$ is said to be \textit{neatly quasiconvex} (see \cite[Definition 4.1]{HAD})
if $h$ is quasiconvex and for every $x\in \mathbb{R}^n$ with $h(x) > \inf h$, the sets $S_{h(x)}
(h)$ and $S^{<}_{h(x)} (h)$ have the same closure (or equivalently, the same
relative interior). As a consequence, a quasiconvex function $h$ is neatly
quasiconvex if and only if every local minimum of $h$ is global minimum (see
\cite[Proposition 4.1]{HAD}). In parti\-cu\-lar, every semistrictly
quasiconvex function is neatly quasiconvex, and every continuous and neatly
quasiconvex function is semistrictly quasiconvex by \cite[Proposition
4.2]{HAD}. The function in \cite[Example 4.1]{HAD} is neatly quasiconvex
without being semistrictly quasiconvex. Recall that
\begin{eqnarray*}
 h ~ {\rm is ~ convex} & \Longleftrightarrow & \epi h ~ {\rm is ~ a ~ convex
 ~ set;}  \\
 h ~ {\rm is ~ quasiconvex} & \Longleftrightarrow & S_{\lambda} (h) ~ {\rm is ~ a
 ~ convex ~ set ~ for ~ all ~ } \lambda \in \mathbb{R}.
\end{eqnarray*}

For algorithmic purposes, the following notions from \cite[Definition 10.27]{BAC} (see also \cite{VI1,VI2}) are useful.

A function $h$ with a convex domain is said to be \textit{strongly convex} (respectively \textit{strongly quasiconvex}), if there exists $\beta \in ]0, + \infty[$ such
that for all $x, y \in \dom \,h$ and all $\lambda \in [0, 1]$, we have
 \begin{align}
  & ~~~~ h(\lambda y + (1-\lambda)x) \leq \lambda h(y) + (1-\lambda) h(x) - \lambda
  (1 - \lambda) \frac{\beta}{2} \lVert x - y \rVert^{2}, \label{strong:convex} \\
  & \left( {\rm respectively} ~ h(\lambda y + (1-\lambda)x) \leq \max\{h(y), h(x)\} -
  \lambda (1 - \lambda) \frac{\beta}{2} \lVert x - y \rVert^{2}. \right)
  \label{strong:quasiconvex}
 \end{align}
 For \eqref{strong:quasiconvex}, sometimes one needs to restrict the value
 $\beta$ to a subset $J$ in $]0, + \infty[$ and then $h$ is said to be {\it strongly quasiconvex for $J$}.

Every strongly convex function is strongly quasiconvex, and eve\-ry strongly
quasiconvex function is semistrictly quasiconvex. Furthermore, a strong\-ly
quasiconvex function has at most one minimizer on a convex set $K \subseteq
\mathbb{R}^n$ that touches its domain (see \cite[Proposition 11.8]{BAC}).

\medskip

A function $h: \mathbb{R}^{n} \rightarrow \overline{\mathbb{R}}$ is said to be
\begin{itemize}
 \item[$(a)$] \textit{supercoercive} if
 \begin{equation}\label{supercoercive}
  \liminf_{\lVert x \rVert \rightarrow + \infty} \frac{h(x)}{\lVert x \rVert}
  = + \infty;
 \end{equation}

 \item[$(b)$] \textit{coercive} if
 \begin{equation}\label{coercive}
  \lim_{\lVert x \rVert \rightarrow + \infty} h(x) = + \infty;
 \end{equation}

 \item[$(c)$] \textit{weakly coercive} if
 \begin{equation}\label{weakly:coercive}
  \liminf_{\lVert x \rVert \rightarrow + \infty} \frac{h(x)}{\lVert x \rVert}
  \geq 0;
 \end{equation}
 \item[$(d)$] \textit{$2$-weakly coercive} if
 \begin{equation}\label{2weakly:coercive}
  \liminf_{\lVert x \rVert \rightarrow + \infty} \frac{h(x)}{\lVert x \rVert^{2}}
  \geq 0.
  \end{equation}
\end{itemize}
Clearly, $(a) \Rightarrow (b) \Rightarrow (c) \Rightarrow (d)$. The function $h(x)
= \sqrt{\lvert x \rvert}$ is coercive without being supercoercive; the function
$h(x) = - \sqrt{\lvert x \rvert}$ is weakly coercive without being coercive
(moreover, it is not even bounded from below). Finally, the function $h(x) =
- \lvert x \rvert$ is $2$-weakly coercive without being weakly coercive. Recall
that $h$ is coercive if and only if $S_{\lambda} (h)$ is a bounded set for every
$\lambda \in \mathbb{R}$. A survey on coercivity notions is \cite{CCA}.

\medskip

The \textit{convex subdifferential} of a proper function $h: \mathbb{R}^{n} \rightarrow \overline{\mathbb{R}}$ at $x \in \mathbb{R}^{n}$ is
\begin{equation} \label{subd:usual}
 \partial h(x) := \left\{ \xi \in \mathbb{R}^{n}: h(y) \geq h(x) + \langle \xi, y - x \rangle, ~ \forall ~ y \in \mathbb{R}^{n}\right\},
\end{equation}
when $x \in \dom \,h$, and $\partial h(x) = \emptyset$ if $x \not\in \dom \,h$. But in case of nonconvex functions (quasiconvex for ins\-tance) the
convex subdifferential is too small and often empty, other subdifferential notions (see \cite{GUT,PLA}) being necessary, like the \textit{Guti\'errez subdifferential} (of $h$ at $x$), defined by
\begin{equation}\label{Guti:sub}
 \partial^{\leq} h(x) := \left\{ \xi \in \mathbb{R}^{n}: ~ h(y) \geq h(x) + \langle
 \xi, y - x \rangle, ~ \forall ~ y \in S_{h(x)} (h)\right\},
\end{equation}
when $x \in \dom \,h$, and $\partial^{\leq} h(x) = \emptyset$ if $x \not\in \dom \,h$, or the \textit{Plastria subdifferential} (of $h$ at $x$), that is
 \begin{equation}\label{Plastria:sub}
  \partial^{<} h(x) := \left\{ \xi \in \mathbb{R}^{n}: ~ h(y) \geq h(x) + \langle
  \xi, y - x \rangle, ~ \forall ~ y \in S^{<}_{h(x)} (h)\right\},
 \end{equation}
when $x \in \dom \,h$, and $\partial^{<} h(x) = \emptyset$ if $x \not\in \dom \,h$. Clearly, $\partial h \subseteq \partial^{\leq} h \subseteq \partial^{<} h$. The
reverse inclusions do not hold as the function $h: \mathbb{R} \rightarrow
\mathbb{R}$ given by $h(x) = \min\{x, \max\{x-1, 0\}\}$ shows (see \cite[page
21]{PE1}). A sufficient condition for equality in this inclusion chain is given in \cite[Proposition
10]{PE1}.

Note that both $\partial^{\leq} h$ and $\partial^{<} h$ are (at any point) either
empty or unbounded, and it holds (see \cite{GUT,PLA,PE1})
\begin{equation}\label{G:charmin}
 0 \in \partial^{<} h(x) \Longleftrightarrow 0 \in \partial^{\leq} h(x)
 \Longleftrightarrow x \in \amin_{\mathbb{R}^{n}}\,h \Longleftrightarrow
 \partial^{\leq} h(x) = \mathbb{R}^{n}.
\end{equation}

We recall the following results originally given in \cite[Theorem 2.3]{PLA}, \cite[Proposition 2.5 and Proposition 2.6]{XRG} and \cite[Theorem 20]{CSE}, respectively.

\begin{lemma}\label{nonempty:subd}
 Let $h: \mathbb{R}^{n} \rightarrow \overline {\mathbb{R}}$ be a proper function. The following results hold.
 \begin{itemize}
  \item[$(a)$] If $h$ is quasiconvex and $L$-Lipschitz, then $\partial^{<} h(x)
  \neq \emptyset$ for all $x \in \mathbb{R}^{n}$. Mo\-re\-over, there exists $\xi
  \in \partial^{<} h(x)$ such that $\lVert \xi \rVert \leq L$.

  \item[$(b)$] If $h$ is neatly quasiconvex and $L$-Lipschitz, then $\partial^{
  \leq} h (x) \neq \emptyset$ for all $x \in \mathbb{R}^{n}$. Moreover, if $u \in \partial^{\leq} h(x)$, $u \neq 0$, then
  $L \frac{1}{\lVert u \rVert}u \in \partial^{\leq} h(x)$.
 \end{itemize}
\end{lemma}

For $\gamma>0$ we define the \textit{Moreau envelope of parameter} ${\gamma}$ of $h$ by
\begin{equation}
 ^{\gamma}h(z) = \inf_{x \in \mathbb{R}^{n}} \left(h(x) + \frac{1}{2\gamma}
 \lVert z-x \rVert^{2}\right).
\end{equation}

The \textit{proximity operator of parameter $\gamma >0$} of a function $h:\mathbb{R}^{n} \rightarrow \overline {\mathbb{R}}$ at $x \in \mathbb{R}^{n}$ is
defined as
\begin{equation}\label{gammah-def}
 {\prox}_{\gamma h}: \mathbb{R}^{n} \rightrightarrows \mathbb{R}^{n},\ {\prox}_{\gamma h}(x)= \amin\limits_{y\in \mathbb{R}} \left\{h(y) + \frac{1}{2 \gamma}
 \|y-x\|^2\right\}.
\end{equation}
When $h$ is proper, convex and lower semicontinuous, ${\prox}_{\gamma h}$ turns out to be a single-valued operator. By a slight abuse of notation, when ${\prox}_{\gamma h}$ is single-valued we write in this paper ${\prox}_{\gamma h}(z)$ (for some $z\in \mathbb{R}^n$) to identify the unique element of the actual set ${\prox}_{\gamma h}(z)$. Moreover, when $\gamma =1 $ we write ${\prox}_{h}$
instead of ${\prox}_{1 h}$.

For studying constrained optimization problems, the use of constrained notions
becomes important since they ask for weaker conditions. Indeed, for instance, the
function $h: \mathbb{R} \rightarrow \mathbb{R}$ given by $h(x) = \min\{\lvert x
\rvert, 2\}$ is convex on $K = [-2, 2]$, but is not convex on $\mathbb{R}$.

For a nonempty set $K$ in $\mathbb{R}^{n}$, by $\partial_{K} h(x), \partial^{\leq
}_{K} h(x)$ and $\partial^{<}_{K} h(x)$, we mean the convex, Guti\'errez and
Plastria subdifferentials of $h$ at $x \in K$ restricted to the set $K$, that is,
\begin{align*}
 & \partial_{K} h(x) := \partial (h+\delta_K)(x) = \left\{ \xi \in \mathbb{R}^{n}:
 ~ h(y) \geq h(x) + \langle \xi, y - x \rangle, ~ \forall ~ y \in K \right\},
\end{align*}
as well as $\partial^{\leq}_{K} h(x) := \partial^{\leq} (h+\delta_K)(x)$ and $\partial^{<}_{K} h(x) := \partial^{<} (h+\delta_K)(x)$.

For $K\subseteq \mathbb{R}^{n}$, a single-valued operator $T: K \rightarrow \mathbb{R}^{n}$ is called
\begin{itemize}
 \item[$(a)$] \textit{monotone} on $K$, if for all $x, y \in K$, we have
  \begin{equation}\label{def:monoT}
   \langle T(x) - T(y), x - y \rangle \geq 0;
  \end{equation}

 \item[$(b)$] \textit{firmly nonexpansive} if for every $x, y \in K$, we have
 \begin{equation}\label{FNE}
  \lVert T(x) - T(y) \rVert^{2} + \lVert (\id - T) (x) - (\id - T) (y) \rVert^{2}
  \leq \lVert x - y \rVert^{2}, ~ \forall ~ x, y \in K,
 \end{equation}
\end{itemize}

According to \cite[Proposition 4.4]{BAC}, $T$ is firmly nonexpansive if and only if
\begin{equation}\label{FNEM}
 \lVert T(x) - T(y) \rVert^{2} \leq \langle x - y, T(x) - T(y) \rangle, ~
 \forall ~ x, y \in K.
\end{equation}
As a consequence, if $T$ is firmly nonexpansive, then $T$ is Lipschitz continuous and monotone.

\section{Prox-convex functions}\label{sec:3}

In this section, we introduce and study a class of functions for which the
nece\-ssa\-ry fundamental properties presented in the introduction are satisfied.

\subsection{Motivation, definition and basic properties}

We begin with the following result, in which we provide a general sufficient condition for the nonemptiness of the values of the proximity operator.

\begin{proposition}\label{pr1}
\textit{Let $h: \mathbb{R}^{n} \rightarrow \overline{\mathbb{R}}$ be a proper,
lower semicontinuous and $2$-weakly coercive func\-tion. Given any $z \in
\mathbb{R}^{n}$, there exists $\overline{x} \in \prox_{h} (z)$.}
\end{proposition}

\begin{proof}
 Given $z \in \mathbb{R}^{n}$, we consider the minimization problem:
 \begin{equation}\label{def:hz}
  \min_{x \in \mathbb{R}^{n}} h_{z} (x) := h(x) + \frac{1}{2} \lVert x - z
  \rVert^{2}.
 \end{equation}
 Since $h$ is lower semicontinuous and $2$-weakly coercive, $h_{z}$ is lower
 semi\-con\-ti\-nuous and coercive by \cite[Theorem 2$(ii)$]{CCA}. Thus, there
 exists $\overline{x} \in \mathbb{R}^{n}$ such that $\overline{x} \in
 {\amin}_{\mathbb{R}^{n}} h_{z}$, i.e., $\overline{x} \in \prox_{h} (z)$.
\end{proof}

One cannot weaken the assumptions of Proposition \ref{pr1} without losing its conclusion.

\begin{remark}\label{exis:is:optimal}
 \begin{itemize}
  \item[$(i)$] Note that every convex function is $2$-weakly coer\-ci\-ve, and
  every bounded from below function is also $2$-weakly coer\-ci\-ve. The function
  $h: \mathbb{R}^{n} \rightarrow \mathbb{R}$ given by $h(x) = - \lvert x \rvert$
  is $2$-weakly coercive, but is neither convex nor bounded from below. However,
  for any $z \in \mathbb{R}^{n}$, $\prox_{h} (z) \neq \emptyset$.

  \item[$(ii)$] The $2$-weak coercivity assumption can not be dropped in the
  general case. Indeed, the function $h:  \mathbb{R} \rightarrow \mathbb{R}$ given by $h(x) = -x^{3}$ is continuous and quasiconvex, but fails to be $2$-weakly coercive and for any $z \in \mathbb{R}$ one has $\prox_{h} (z) = \emptyset$.
 \end{itemize}
\end{remark}

Next we characterize the existence of solution in the definition of the proximity operator.

\begin{proposition}\label{pr2}
 \textit{Let $h: \mathbb{R}^{n} \rightarrow \overline{\mathbb{R}}$ be a proper
 function. Given any $z \in \mathbb{R}^{n}$, one has
 \begin{equation}\label{prox:nonconvex}
\overline{x} \in {\prox}_{h} (z) ~ \Longleftrightarrow ~ h(\overline{x}) -  h(x) \leq \frac{1}{2} \langle \overline{x} + x - 2z, x - \overline{x} \rangle,   ~ \forall  x \in \mathbb{R}^{n}.
 \end{equation}}
\end{proposition}

\begin{proof}
Let $z\in \mathbb{R}^{n}$. One has
\begin{align*}
  \overline{x} \in {\prox}_{h} (z) & \Longleftrightarrow \, h(\overline{x}) +
  \frac{1}{2} \lVert \overline{x} - z \rVert^{2} \leq h(x) + \frac{1}{2} \lVert
  x - z \rVert^{2},\ \forall \, x \in \mathbb{R}^{n} \\
  & \Longleftrightarrow \, h(\overline{x}) - h(x) \leq \frac{1}{2} \lVert x - z
  \rVert^{2} - \frac{1}{2} \lVert \overline{x} - z \rVert^{2}m \ \forall \, x \in \mathbb{R}^{n} \\
  & \Longleftrightarrow \, h(\overline{x}) - h(x) \leq \langle \overline{x} - z,
  x - \overline{x} \rangle + \frac{1}{2} \lVert x - \overline{x} \rVert^{2}, \  \forall \, x \in \mathbb{R}^{n} \\
  & \Longleftrightarrow \, h(\overline{x}) - h(x) \leq \frac{1}{2} \langle
  \overline{x} + x - 2z, x - \overline{x} \rangle, \ \forall \, x \in \mathbb{R}^{n}.
 \end{align*}
\end{proof}

Relation \eqref{prox:nonconvex} is too general for providing convergence results
for proximal point type algorithms while relation \eqref{fundamental} has proven to be
extremely useful in the convex case. Motivated by this, we introduce the class of
\textit{prox-convex functions} below. In the following, we write
\begin{equation}\label{prx}
 {\prox}_{h} (K, z) := {\prox}_{(h + \delta_{K})} (z).
\end{equation}

Note that closed formulae for the proximity operator of a sum of functions in terms
of the proximity operators of the involved functions are known only in the convex
case and under demanding hypotheses, see, for instance, \cite{ABC}. However,
constructions like the one in \eqref{prx} can be found in the literature on proximal
point methods for solving different classes of (nonconvex) optimization problems,
take for instance \cite{BCF, GRN}.

\begin{definition}\label{all:functions}
 Let $K$ be a closed set in $\mathbb{R}^{n}$ and $h: \mathbb{R}^{n} \rightarrow
 \overline{\mathbb{R}}$ be a proper func\-tion such that $K \cap \dom \,h \neq
 \emptyset$. We say that $h$ is \textit{prox-convex on $K$} if there exists
 $\alpha > 0$ such that for every $z \in K$, $\prox_{h} (K, z) \neq \emptyset$, and
  \begin{equation}\label{prox:all}
  \overline{x} \in {\prox}_{h} (K, z) ~ \Longrightarrow ~ h(\overline{x}) - h(x)
  \leq \alpha \langle \overline{x} - z, x - \overline{x} \rangle, ~ \forall
  x \in K.
 \end{equation}
The set of all prox-convex function on $K$ is denoted by $\Phi(K)$, and the scalar
$\alpha > 0$ for which \eqref{prox:all} holds is said to be the \textit{prox-convex value} of the function $h$ on $K$. When $K=\mathbb{R}^{n}$ we say that $h$ is \textit{prox-convex}.
\end{definition}

\begin{remark}
\begin{itemize}
 \item[$(i)$] One can immediately notice that \eqref{prox:all} is equivalent to a weaker version
of \eqref{fundamental}, namely
$$ \overline{x} \in {\prox}_{h} (K, z) ~ \Longrightarrow ~
z - \overline{x} \in \partial \left(\frac{1}{\alpha} (h+ \delta_K)\right)(\overline{x}).$$

 \item[$(ii)$] The scalar $\alpha > 0$ for which \eqref{prox:all} holds needs not
  be unique. Indeed, if $h$ is convex, then $\alpha = 1$ by Proposition
  \ref{include:convex}. However, due to the con\-vexity of $h$, $ \langle
  \overline{x} - z, x - \overline{x} \rangle \geq 0$. Hence, $\overline{x} \in
  \prox_{h} (K, z)$ implies that
  \begin{align*}
   h(\overline{x}) - h(x) \leq  \langle \overline{x} - z, x - \overline{x} \rangle
   \leq \gamma \langle \overline{x} - z, x - \overline{x} \rangle, ~ \forall ~
   \gamma \geq 1,\ ~ \forall ~ x \in K.
  \end{align*}
 Note however that a similar result does not necessarily hold in general, as
 $ \langle \overline{x} - z, x - \overline{x} \rangle$ might be negative.

 \item[$(iii)$] Note also that, at least from the computational point of view,
 an exact value of $\alpha$ needs not be known, as one can see in Section \ref{sec:4}.
\end{itemize}
\end{remark}

In the following statement we see that in the left-hand side of \eqref{prox:all} one can replace the element-of symbol with equality since the proximity operator of a proper prox-convex function is single-valued and also firmly nonexpansive.

\begin{proposition}\label{single-val}
Let $K$ be a closed set in $\mathbb{R}^{n}$ and $h: \mathbb{R}^{n} \rightarrow \overline{\mathbb{R}}$ a proper prox-convex function on $K$ such that $K \cap \dom \,h \neq \emptyset$. Then the map $z \rightarrow \prox_{h} (K, z)$ is single-valued and firmly nonexpansive.
\end{proposition}

\begin{proof}
Suppose that $h$ is a prox-convex function with prox-convex value $\alpha > 0$ and assume that for some $z\in K$ one has $\{\overline{x}_{1}, \overline{x}_{2}\} \subseteq \prox_{h} (K, z)$. Then
 \begin{align}
  & h(\overline{x}_{1}) - h(x) \leq \alpha \langle \overline{x}_{1} - z, x - \overline{x}_{1} \rangle, ~ \forall  x \in K, \label{ec:1}\\
  & h(\overline{x}_{2}) - h(x) \leq \alpha \langle \overline{x}_{2} - z, x - \overline{x}_{2} \rangle, ~ \forall  x \in K. \label{ec:2}
 \end{align}
Take $x = \overline{x}_{2}$ in \eqref{ec:1} and $x = \overline{x}_{1}$ in
\eqref{ec:2}. By adding the resulting equations, we get
$$0 \leq \alpha \langle \overline{x}_{1} - \overline{x}_{2}, \overline{x}_{2} - z +
z -\overline{x}_{1} \rangle = -\alpha \|\overline{x}_{1} - \overline{x}_{2}\|^2
\leq 0.$$
Hence, $\overline{x}_{1} = \overline{x}_{2}$, consequently $\prox_{h} (K, \cdot)$ is single-valued.

Let $z_{1}, z_{2} \in K$ and take $\overline{x}_{1} \in \prox_{h} (K, z_{1})$ and $\overline{x}_{2} \in \prox_{h} (K, z_{2})$. One has
 \begin{align}
  & h(\overline{x}_{1}) - h(x) \leq \alpha \langle \overline{x}_{1} - z_{1}, x -
  \overline{x}_{1} \rangle, ~ \forall ~ x \in K, \label{equa:1}\\
  & h(\overline{x}_{2}) - h(x) \leq \alpha \langle \overline{x}_{2} - z_{2}, x -
  \overline{x}_{2} \rangle, ~ \forall ~ x \in K. \label{equa:2}
 \end{align}
Taking $x = \overline{x}_{2}$ in \eqref{equa:1} and $x = \overline{x}_{1}$ in \eqref{equa:2} and adding them, we have
$$\lVert \overline{x}_{1} - \overline{x}_{2} \rVert^{2} \leq \langle z_{1} - z_{2}, \overline{x}_{1} - \overline{x}_{2} \rangle.$$
Hence, by \cite[Proposition 4.4]{BAC}, $\prox_{h} (K, \cdot)$ is firmly nonexpansive.
\end{proof}

Next we show that every lower semicontinuous and convex function is prox-convex.

\begin{proposition}\label{include:convex}
 \textit{Let $K$ be a closed and convex set in $\mathbb{R}^{n}$ and $h:\mathbb{R}^{n}
 \rightarrow \overline{\mathbb{R}}$ be a proper and lower semicontinuous function
 such that $K \cap \dom \,h \neq \emptyset$. If $h$ is convex on $K$, then $h \in \Phi(K)$
 with $\alpha=1$.}
\end{proposition}

\begin{proof}
Since $h$ is convex, the function $x \mapsto h(x) + ({\beta}/{2}) \lVert
 z - x \rVert^{2}$ is strongly convex on $K$ for all $\beta > 0$ and all $z \in
 K$, in particular, for $\beta = 1$. Thus $\prox_{h} (K, z)$ contains exactly
 one element, say $\overline{x} \in \mathbb{R}^{n}$. It follows from
 \cite[Proposition 12.26]{BAC} that $z - \overline{x} \in \partial h
 (\overline{x})$, so relation  \eqref{prox:all} holds for $\alpha = 1$. Therefore,
 $h \in \Phi (K)$.
\end{proof}

Prox-convexity goes beyond convexity as shown below.

\begin{example}\label{main:example}
 Let us consider $K := [0, 1]$ and the continuous and real-valued function $h: K
 \rightarrow \mathbb{R}$ given by $h(x) = - x^{2} - x$. Note that
 \begin{itemize}
  \item[$(i)$] $h$ is strongly quasiconvex without being convex on $K$  (take
  $\beta = 1$);

  \item[$(ii)$] For all $z \in K$, $\prox_{h} (K, z) = {\amin}_{K} h =  \{1\}$;

  \item[$(iii)$] $\partial^{\leq}_{K} h(1) = \mathbb{R}^{n}$ since, by $(ii)$, $S_{h(1)} (h) = \{1\}$, i.e., $\partial^{\leq} h(1) = \mathbb{R}^{n}$ by \eqref{G:charmin};

  \item[$(iv)$] $h$ satisfies con\-di\-tion \eqref{prox:all} for all $\alpha > 0$;

  \item[$(v)$] $h \in \Phi(K)$;

  \item[$(vi)$] more generally, $h \in \Phi([0, r])$ for all $r>0$.
 \end{itemize}
\end{example}

In order to formulate a reverse statement of Proposition \ref{include:convex}, we
note that if $h:\mathbb{R}^{n}\to \overline{\mathbb{R}}$ is a lower semicontinuous and prox-convex function on some set $K \cap \dom \,h \neq \emptyset$ which satisfies
 \eqref{prox:all} for $\alpha = 1$, then $h$ is not ne\-ce\-ssa\-ri\-ly convex. Indeed, the function in Example \ref{main:example} satisfies
\eqref{prox:all} for all $\alpha > 0$, but it is not convex on $K = [0, 1]$.

In the following example, we show that lower semicontinuity is not a necessary condition for prox-convexity. Note also that although the proximity operator of the function mentioned in Remark \ref{exis:is:optimal}$(ii)$ is always empty, this is no longer the case when restricting it to an interval.

\begin{example}\label{no:lsc}
Take $n \geq 3$, $K_{n} := [1, n]$ and the function $h_{n}: K_{n} \rightarrow
 \mathbb{R}$ given by
 $$
  h_{n} (x) = \left \{
  \begin{array}[c]{cl}
  1- x^{3}, & \mathrm{if} ~ 1 \leq x \leq 2, \\
  1- x^{3} - k, & \mathrm{if} ~ k < x \leq k+1, ~ k \in \{2, \ldots, n-1\}.
  \end{array}
  \right.
 $$
Note that $h_{n}$ is neither convex nor lower semicontinuous, but it is quasiconvex on $K_{n}$.
Due to the discontinuity of $h_{n}$, the function $f_{n} (x) = h_{n} (x) +
 ({1}/{2}) \lVert x \rVert^{2}$ is neither convex nor lower semicontinuous on $K_{n}$, hence $h_{n}$ is not $c$-weakly convex (in the sense of \cite{HLO}) either
 and also its subdifferential is not hypomonotone (as defined in \cite{COP, IPS, PEN}). However, for any $z \in K_{n}$, $\prox_{h_{n}} (K_n, z) = \{n\}$, and
 $\partial^{\leq}_{K_{n}} h_{n} (n) = K_{n}$. Therefore, $h_{n} \in \Phi(K_{n})$.
\end{example}

Another example of a prox-convex function that is actually (like the one in Example \ref{main:example}) both concave and DC follows.
\begin{example}\label{ln}
Take $K=[1, 2]$ and $h: (0, +\infty) \to \mathbb{R}$ defined by $h(x)= 5x + \ln (1+10x)$. As specified in \cite{MUQ}, both the prox-convex function presented in Example \ref{main:example} and this one represent cost functions considered in oligopolistic equilibrium problems, being thus relevant for studying also from a practical point of view.
One can show that $\prox_{h} (K, z) = {\amin}_{K} h =  \{1\}$ for all $z \in K$ and \eqref{prox:all} is fulfilled for $\alpha \in (0, 5)$.
\end{example}

\begin{remark}
\begin{itemize}
\item[(i)] One can also construct examples of $c$-weakly convex functions (for some $c>0$) that are not prox-convex, hence these two classes only contain some common elements without one of them being completely contained in the other.

\item[(ii)] While Examples \ref{main:example} and \ref{ln} exhibit prox-convex functions that are also DC, the prox-convex functions presented in Example \ref{no:lsc} are not DC. Examples of DC functions that are not prox-convex can be constructed as well, consequently, like in the case of $c$-weakly convex functions, these two classes only contain some common elements without one of them being completely contained in the other. Note moreover that different to the literature on algorithms for DC optimization problems (see, for instance, \cite{AFV, BBO}) where usually only critical points (and not optimal solutions) of such problems are determinable, for the DC functions that are also prox-convex proximal point methods are capable of delivering global minima (on the considered sets).

\item[(iii)] The remarkable properties of the Kurdyka-\L ojasiewicz (K{\L}) functions made them a standard tool when discussing proximal point type algorithms for nonconvex functions. As their definition requires proper closedness and the prox-convex functions presented in Example \ref{no:lsc} are not closed, one can conclude that the class of prox-convex functions is broader in this sense than the one of K{\L} ones. Similarly one can note that prox-convexity is not directly related to  hypomonotonicity of subdifferentials (see \cite{COP, IPS,PEN}, respectively).

\item[(iv)] At least due to the similar name, a legitimate question is whether the notion of prox-convexity is connected in any way with the
 prox-regularity (cf. \cite{COP, LEW, PEN}). While the latter asks a function to be locally lower semicontinuous around a given point, the notion we introduce in this
 work does not assume any topological properties on the involved function. Another difference with respect to this notion can be noticed in Section \ref{sec:4}, where
 we show that the classical proximal point algorithm remains convergent towards a minimum of the function to be minimized even if this lacks convexity, but is
 prox-convex. On the other hand, the iterates of the modified versions of the proximal point method employed for minimizing prox-regular functions converge
 towards critical points of the latter. Last but not least note that, while in the mentioned works one uses tools specific to nonsmooth analysis such as generalized
 subdifferentials, in this paper we employ the convex subdifferential and some subdifferential notions specific to quasiconvex functions.

\end{itemize}
\end{remark}

Necessary and sufficient hypotheses for condition \eqref{prox:all} are given below.

\begin{proposition}\label{char:proxconvex}
 Let $K$ be a closed set in $\mathbb{R}^{n}$ and $h: \mathbb{R}^{n} \rightarrow \overline{\mathbb{R}}$ be a proper, lower semicontinuous and prox-convex function such that $K \cap \dom \,h \neq \emptyset$. Let $\alpha > 0$ be the prox-convex value of $h$ on $K$,
 and $z \in K$. Consider the following assertions
  \begin{itemize}
   \item[$(a)$]  $\prox_{h} (K, z) = \{\overline{x}\}$; 

   \item[$(b)$] $z - \overline{x} \in \partial_{K} \left(\frac {1}{\alpha}h\right)(\overline{x})$;

   \item[$(c)$] $(\frac{1}{\alpha}h)_{z} (\overline{x}) - (\frac{1}{\alpha}h)_{z}
   (x) \leq - \frac{1}{2} \lVert x - \overline{x} \rVert^{2}$ for all $x \in K$;

   \item[$(d)$] $\overline{x} \in \prox_{\frac{1}{\alpha}h} (K, z)$.
  \end{itemize}
  Then
  $$(a) ~ \Longrightarrow ~ (b) ~ \Longleftrightarrow ~ (c) ~ \Longrightarrow
  ~ (d).$$
  If $\alpha = 1$, then $(d)$ implies $(a)$ and all the statements are equivalent.
\end{proposition}

\begin{proof}
 $(a) \Rightarrow (b)$: By definition of prox-convexity.

 $(b) \Leftrightarrow (c)$: One has
  \begin{align}
  & z - \overline{x} \in \partial_{K} ( \frac{1}{\alpha} h)(\overline{x}) ~
  \Longleftrightarrow ~ (\frac{1}{\alpha} h) (\overline{x}) - (\frac{1}{\alpha}
  h)(x) \leq \langle \overline{x} - z, x - \overline{x} \rangle, ~ \forall  x
  \in K \notag \\
  & \Longleftrightarrow ~ \frac{1}{\alpha} h(\overline{x}) - \frac{1}{\alpha}
  h(x) \leq \frac{1}{2} \lVert z - x \rVert^{2} - \frac{1}{2} \lVert z -
  \overline{x} \rVert^{2} - \frac{1}{2} \lVert x - \overline{x} \rVert^{2}, ~
  \forall  x \in K \notag \\
  & \Longleftrightarrow ~ \frac{1}{\alpha}h(\overline{x}) + \frac{1}{2} \lVert z
  - \overline{x} \rVert^{2} - \frac{1}{\alpha}h(x) - \frac{1}{2} \lVert z - x
  \rVert^{2} \leq - \frac{1}{2} \lVert x - \overline{x} \rVert^{2}, ~ \forall
  x \in K \notag \\
  & \Longleftrightarrow ~ (\frac{1}{\alpha} h)_{z} (\overline{x}) - (\frac{1}{
  \alpha} h)_{z} (x) \leq - \frac{1}{2} \lVert x - \overline{x} \rVert^{2}, ~
  \forall  x \in K. \label{third:impli}
 \end{align}
 $(c) \Rightarrow (d)$: As $- ({1}/{2}) \lVert x - \overline{x} \rVert^{2}
 \leq 0$ for all $x \in K$, \eqref{third:impli} yields
 $\overline{x}\in \prox_{(1/\alpha) h} (K, z)$.

 When $\alpha=1$, the implication $(d) \Rightarrow (a)$ is straightforward.
\end{proof}

\begin{remark}
 \begin{itemize}
  \item[$(i)$] It follows from Proposition \ref{char:proxconvex}$(d)$ that if $h$ is
  prox-convex on $K$ with prox-convex value $\alpha > 0$, then the function
  $(1/\alpha) h$ is also prox-convex on $K$ with prox-convex value $1$.
  Moreover, $\prox_{(1/\alpha) h} = \prox_{h}$.

  \item[$(ii)$] Assertions $(d)$ and $(a)$ of Proposition \ref{char:proxconvex}
  are not equivalent in the general case. Indeed, let us consider $h$ and $K$ as
  in Example \ref{main:example} and $\alpha = 10$. Take $z = 0$. Then
  $((1/10) h)_{0} (x) = (2/5) x^{2} - (1/10) x$ and
  $\prox_{(1/10) h} (K, 0) = \{1/8\}$ while $\prox_{h} (K, 0)= \{1\}$.
 \end{itemize}
\end{remark}

If $h$ is prox-convex with prox-convex value $\alpha$, then we know that
$\prox_{(1/\alpha) h} = \prox_{h}$ is a singleton, hence
\begin{equation}
 ^{\frac{1}{\alpha}}h(z) = \min_{x \in K} \left(h(x) + \frac{\alpha}{2} \lVert z - x
 \rVert^{2} \right) = h ({\prox}_{h} (z)) + \frac{\alpha}{2} \lVert z -
 {\prox}_{h}(z) \rVert.
\end{equation}

Consequently, $^{1/\alpha}h(z)\in \mathbb{R}$ for all $z\in \mathbb{R}^{n}$.
Furthermore, we have the fo\-llo\-wing statements.

\begin{proposition}
Let $h:\mathbb{R}^{n} \rightarrow \overline{\mathbb{R}}$ be proper, lower semicontinuous and prox-convex with prox-convex value $\alpha>0$ on a closed set $K \subseteq  \mathbb{R}^{n}$ such that $K \cap \dom \,h \neq \emptyset$. Then $^{1/\alpha}h: \mathbb{R}^{n} \rightarrow \mathbb{R}$
 is Fr\'{e}chet differentiable everywhere and
\begin{equation}\label{gradient:formula}
\nabla({^{\frac{1}{\alpha}}}h)=\alpha\left(\id-{\prox}_{\frac{1}{\alpha}h}\right),
\end{equation}
is $\alpha$-Lipschitz continuous.
\end{proposition}

\begin{proof}
 Let $x, y \in K$ with $x \neq y$. Set $\gamma = 1/\alpha$,
 $\overline{x} = \prox_{h} (K, x)$ and $\overline{y} = \prox_{h} (K, y)$.
As $h$ is prox-convex with prox-convex value $\alpha$, we have
\begin{align*}
 & h(z) - h(\overline{x}) \geq \alpha \langle x - \overline{x}, z - \overline{x}
 \rangle \, \forall \, z \in K \, \Longrightarrow \, h(\overline{y})
 - h(\overline{x}) \geq \frac{1}{\gamma} \langle x - \overline{x}, \overline{y} -
 \overline{x} \rangle.
\end{align*}

From \eqref{gammah-def}, we get
\begin{align}
 ^{\gamma}h(y)-\,^{\gamma}h(x) & = h(\overline{y}) - h(\overline{x}) +
 \frac{1}{2\gamma} \left( \lVert y - \overline{y} \rVert^{2} - \lVert x -
 \overline{x} \rVert^{2}\right) \notag \\
 & \geq \frac{1}{2 \gamma} \left( 2 \langle x - \overline{x}, \overline{y} -
 \overline{x} \rangle + \lVert y - \overline{y} \rVert^{2} - \lVert x -
 \overline{x} \rVert^{2}\right) \notag \\
 &  =\frac{1}{2 \gamma} \left( \lVert y - \overline{y} - x + \overline{x}
 \rVert^{2} + 2 \langle y-x, x - \overline{x} \rangle \right) \notag \\
 & \geq \frac{1}{\gamma} \langle y - x, x - \overline{x} \rangle. \label{F1}
\end{align}

Exchanging above $x$ with $y$ and $\overline{x}$ with $\overline{y}$, one gets
\begin{equation}\label{F2}
 ^{\gamma}h(x) - \,^{\gamma}h(y) \geq \frac{1}{\gamma} \langle x - y, y -
\overline{y} \rangle.
\end{equation}

It follows from equations \eqref{F1} and \eqref{F2} that
\begin{align*}
 0  &  \leq \,^{\gamma}h(y) - \,^{\gamma} h(x) - \frac{1}{\gamma} \langle y - x,
 x - \overline{x} \rangle \\
 & \leq - \frac{1}{\gamma} \langle x - y, y - \overline{y} \rangle -
 \frac{1}{\gamma} \langle y - x, x - \overline{x} \rangle \\
 & = \frac{1}{\gamma} \lVert y - x \rVert ^{2} + \frac{1}{\gamma} \langle y - x,
 \overline{x} - \overline{y} \rangle.
\end{align*}

As $\prox_{K, h}$ is firmly nonexpansive on $K$, $\langle y - x, \overline{y} -
\overline{x} \rangle \geq \lVert \overline{y} - \overline{x} \rVert^{2} \geq 0$,
then
\begin{align*}
 & 0 \leq \,^{\gamma}h(y) - \,^{\gamma}h(x) - \frac{1}{\gamma} \langle y - x,
 x - \overline{x} \rangle \leq \frac{1}{\gamma} \lVert y - x \rVert^{2} \\
 & \Longrightarrow ~ \lim_{y \rightarrow x} \frac{^{\gamma}h(y) -
 \,^{\gamma}h(x) - \frac{1}{\gamma} \langle y - x, x - \overline{x} \rangle}{
 \lVert y - x \rVert}  = 0.
\end{align*}

Thus, $^{1/\alpha}h$ is Fr\'{e}chet differentiable at every $x\in \mathbb{R}^{n}$,
and $\nabla(^{1/\alpha}h)=\alpha(  \id -$ $\prox_{h})$. Since $\prox_{h}$ is
firmly nonexpansive, $\id-\prox_{h}$ is also firmly
nonexpansive, so $\nabla(^{1/\alpha}h)$ is $\alpha$-Lipschitz continuous.
\end{proof}

\subsection{Strongly G-subdifferentiable functions}

Further we introduce and study a class of quasiconvex functions whose lower semicontinuous members are prox-convex.

\begin{definition}\label{def:proxqcx}
 Let $K$ be a closed and convex set in $\mathbb{R}^{n}$ and $h: \mathbb{R}^{n}
 \rightarrow \overline{\mathbb{R}}$ be a proper and lower semicontinuous
 function such that $K \cap \dom \,h \neq \emptyset$. We call $h$ strongly G-subdifferentiable on $K$ if
 \begin{itemize}
  \item[$(a)$] $h$ is strongly quasiconvex on $K$ for some $\beta \in [1, + \infty[$;

  \item[$(b)$] for each $z \in K$ there exists $\overline{x}\in \mathbb{R}^{n}$ such that  $\prox_{h} (K, z) = \{\overline{x}\}$ and
  \begin{equation}\label{quali:cond}
  \frac{1}{2} (z - \overline{x}) \in \partial^{\leq}_{K} h(\overline{x}).
  \end{equation}
 \end{itemize}
\end{definition}

Next we show that a lower semicontinuous and strongly G-subdifferentiable function on $K$ is prox-convex.

\begin{proposition}\label{prox:qcx-phi}
 \textit{Let $K$ be a closed and convex set in $\mathbb{R}^{n}$ and $h: \mathbb{R}^{n}
 \rightarrow \overline{\mathbb{R}}$ be a proper and lower semicontinuous function
 such that $K \cap \dom \,h \neq \emptyset$. If $h$ is strongly G-subdifferen\-tia\-ble
 on $K$, then $h \in \Phi (K)$.}
\end{proposition}

\begin{proof}
 Let $h$ be a lower semicontinuous and strongly G-subdifferentiable function. Then for every $z
 \in K$, there exists $\overline{x} \in K$ with $\overline{x} = \prox_{h}
 (K, z)$. Hence, given any $y \in K$, we take $y_{\lambda} = \lambda y + (1 -
 \lambda) \overline{x}$ with $\lambda \in [0, 1]$. Thus, by the definition of the proximity operator and the strong quasiconvexity of $h$ on $K$ for some $\beta \geq 1$, we have
 \begin{align*}
  h(\overline{x}) & \leq h(\lambda y + (1 - \lambda) \overline{x}) + \frac{1}{2}
  \lVert \lambda y + (1 - \lambda) \overline{x} - z \rVert^{2} - \frac{1}{2}
  \lVert \overline{x} - z \rVert^{2}\\
  & = h(\lambda y + (1 - \lambda) \overline{x}) + \lambda \langle \overline{x} -
  z, y - \overline{x} \rangle + \lambda^{2} \frac{1}{2} \lVert y - \overline{x} \rVert^{2} \\
  & \leq \max\{h(y), h(\overline{x})\} + \lambda \langle \overline{x} - z, y -
  \overline{x} \rangle + \frac{\lambda}{2}(\lambda \beta + \lambda - \beta) \lVert
  y - \overline{x} \rVert^{2}.
 \end{align*}
 We have two possible cases.
 \begin{itemize}
  \item[$(i)$] If $h(y) > h(\overline{x})$, then
  $$h(\overline{x}) - h(y) \leq \lambda \langle \overline{x} - z, y - \overline{x}
   \rangle + \frac{\lambda}{2}(\lambda \beta + \lambda - \beta) \lVert y -
   \overline{x} \rVert^{2},   \ \forall  \lambda \in [0, 1].$$
Hence, for $\lambda = 1/2$ and since $\beta \geq 1$, one has
  \begin{align*}
   h(\overline{x}) - h(y) & \leq \frac{1}{2} \langle \overline{x} - z, y -
  \overline{x} \rangle + \frac{1}{4} (\frac{1}{2} - \frac{\beta}{2}) \lVert y -
  \overline{x} \rVert^{2} \\
  & \leq \frac{1}{2} \langle \overline{x} - z, y - \overline{x} \rangle, \ \forall
  y \in K \backslash S_{h(\overline{x})} (h).
  \end{align*}

  \item[$(ii)$] If $h(y) \leq h(\overline{x})$, then $y \in S_{h(\overline{x})}
  (h)$, it follows from Definition \ref{def:proxqcx}$(b)$ that
\end{itemize}
  \begin{align*}
   \frac{1}{2} (z - \overline{x}) \in \partial^{\leq}_{K} h(\overline{x})
   \Longleftrightarrow h(\overline{x}) - h(y) \leq \frac{1}{2} \langle \overline{x}
   - z, y - \overline{x} \rangle, \ \forall y \in K \cap S_{h(\overline{x})} (h).
  \end{align*}
Therefore, it follows that $h$ satisfies \eqref{prox:all} for $\alpha = {1}/{2}$, i.e., $h \in \Phi(K)$.
\end{proof}

\begin{remark}\label{validity}
 \begin{itemize}
  \item[$(i)$] When $h:\mathbb{R}^{n}\to \overline{\mathbb{R}}$ is lower semicontinuous and strongly quasiconvex, as strongly quasiconvex functions are
  semistrictly quasiconvex, $h$ is quasiconvex and every local minimum of $h$ is a
  global minimum, too, so $h$ is neatly quasiconvex, i.e., $\partial^{<} h =
  \partial^{\leq} h$ (see \cite[Pro\-po\-si\-tion 9]{PE1}). Therefore, we can
  replace $\partial^{\leq}_{K} h$ by $\partial^{<}_{K} h$ in condition
  \eqref{quali:cond}.

  \item[$(ii)$] Strongly $G$-subdifferentiable functions are not necessarily convex as the function in Example \ref{main:example} shows.
 \end{itemize}
\end{remark}




A family of prox-convex functions that are not strongly $G$-subdifferentiable can be found in Remark \ref{rem3.6}, see also Example \ref{no:lsc}.

Now, we study lower semicontinuous strongly quasiconvex functions for which the Gutierr\'ez subdifferential is nonempty. To that end, we first recall the
fo\-llo\-wing definitions (adapted after \cite[Definition 3.1]{CFZ}).

\begin{definition}
 Let $K$ be a nonempty set in $\mathbb{R}^{n}$ and $h: \mathbb{R}^{n} \rightarrow \overline{\mathbb{R}}$ with $K \cap \dom \,h \neq \emptyset$. We say that $h$ is
 \begin{itemize}
  \item[$(a)$] \textit{$\inf$-compact} on $K$ if for all $\overline{x} \in K$,  $S_{h(\overline{x})} (h) \cap K$ is compact;

 \item[$(b)$] \textit{$\alpha$-quasiconvex} at $\overline{x} \in K$ $(\alpha \in
  \mathbb{R})$, if there exist $\rho > 0$ and $e \in \mathbb{R}^{n}$, $\lVert e
  \rVert = 1$, such that
  \begin{equation}\label{CFZ:condition}
   y \in K \cap \mathbb{B} (\overline{x}, \rho) \cap S_{h(\overline{x})} (h)
   \Longrightarrow ~ \langle y - \overline{x}, e \rangle \geq \alpha \lVert y
   - \overline{x} \rVert^{2};
 \end{equation}
  \item[$(c)$] \textit{positively quasiconvex} on $K$ if for any $\overline{x}$
  there exists $\alpha(\overline{x}) > 0$ such that $h$ is
  $\alpha(\overline{x})$-quasiconvex on $S_{h(\overline{x})} (h)$.
 \end{itemize}
\end{definition}

The following result presents a connection between strongly quasiconvex functions
and positively quasiconvex ones.

\begin{proposition}\label{strong:qcx}
 Let $h: \mathbb{R}^{n} \rightarrow \mathbb{R}$ be a strongly quasiconvex
 function, $\overline{x} \in \mathbb{R}^{n}$ and $\alpha > 0$. Then the following
 assertions hold
 \begin{itemize}
  \item[$(a)$] If $\xi \in \partial \left(({1}/{\alpha}) h\right)(\overline{x})$, then
  \begin{equation}
   \langle \xi, y - \overline{x} \rangle \leq - \frac{\beta}{2 \alpha} \lVert y -
   \overline{x} \rVert^{2}, ~ \forall  y \in S_{h(\overline{x})} (h).
  \end{equation}

  \item[$(b)$] If $\xi \in \partial^{\leq} h(\overline{x})$, then
  \begin{equation}
   \langle \xi, y - \overline{x} \rangle \leq - \frac{\beta}{2} \lVert y -
   \overline{x} \rVert^{2}, \ \forall  y \in S_{h(\overline{x})} (h).
   \end{equation}
  \end{itemize}
 As a consequence, in both cases, $h$ is positively quasiconvex on
 $\mathbb{R}^{n}$.
\end{proposition}

\begin{proof}
 The proofs are similar, so we only show $(a)$. Take $\overline{x} \in
 \mathbb{R}^{n}$
 and $\xi \in \partial\left(({1}/{\alpha}) h\right)(\overline{x})$. Then,
 $$\alpha \langle \xi, z - \overline{x} \rangle \leq h(z) - h(\overline{x}), \ \forall  z \in \mathbb{R}^{n}.$$
 Take $y \in S_{h(\overline{x})} (h)$ and $z = \lambda y + (1 - \lambda)
 \overline{x}$ with $\lambda \in [0, 1]$. Then
 \begin{align}
  \lambda \alpha \langle \xi, y - \overline{x} \rangle & \leq h(\lambda y + (1 -
  \lambda) \overline{x}) - h(\overline{x}) \notag \\
  & \leq \max\{h(y), h(\overline{x})\} - \lambda (1 - \lambda) \frac{\beta}{2}
  \lVert y - \overline{x} \rVert^{2} - h(\overline{x}) \notag \\
  & = - \lambda (1 - \lambda) \frac{\beta}{2} \lVert y - \overline{x} \rVert^{2}.
  \notag
 \end{align}
 Then, for every $y \in S_{h(\overline{x})} (h)$, by dividing by $\lambda > 0$
 and taking the limit when $\lambda$ descends towards $0$, we have
\begin{align}
 \langle \xi, y - \overline{x} \rangle \leq \lim_{\lambda \downarrow 0} \left(-
 (1 - \lambda) \frac{\beta}{2 \alpha} \lVert y - \overline{x} \rVert^{2}\right) =
 - \frac{\beta}{2\alpha} \lVert y - \overline{x} \rVert^{2}. \notag
\end{align}

Now, since $h$ is strongly quasiconvex, ${\amin}_{\mathbb{R}^{n}} h$ has at
most one point. If $\overline{x} \in {\amin}_{\mathbb{R}^{n}} h$, then
condition \eqref{CFZ:condition} holds immediately. If $\overline{x} \not\in
{\amin}_{\mathbb{R}^{n}} h$, then $\xi \neq 0$, i.e., condition
\eqref{CFZ:condition} holds for ${\beta}/{(2 \alpha \lVert \xi \rVert)} > 0$.

Therefore, $h$ is positively quasiconvex on $\mathbb{R}^{n}$.
\end{proof}

As a consequence, we have the following result.

\begin{corollary}\label{cor3.1}
 Let $h: \mathbb{R}^{n} \rightarrow \mathbb{R}$ be a lower semicontinuous and
 strongly quasiconvex function with $\beta = 1$, let $z \in \mathbb{R}^{n}$
 and $\overline{x} \in \prox_{h} (z)$. If there exists $\xi \in \partial^{\leq}
 h(\overline{x})$ such that
 \begin{equation}
  h_{z} (\overline{x}) - h_{z} (x) \leq \langle \xi, y - \overline{x} \rangle, \  \forall  y \in S_{h(\overline{x})} (h),
 \end{equation}
 then $h$ is prox-convex on its sublevel set at the height $h(\overline{x})$,
 i.e., $h \in \Phi(S_{h(\overline{x})} (h))$.
\end{corollary}

\begin{proof}
 If $\xi \in \partial^{\leq} h(\overline{x})$, and since $h$ is lower
 semicontinuous and strongly quasiconvex with $\beta = 1$, then by
 Proposition \ref{strong:qcx}$(b)$, we have
 \begin{align}
  & h_{z} (\overline{x}) - h_{z} (x) \leq \langle \xi, y - \overline{x} \rangle
  \leq - \frac{1}{2} \lVert y - \overline{x} \rVert^{2}, \ \forall  y \in
  S_{h(\overline{x})} (h), \notag \\
  & \Longrightarrow ~ h(\overline{x}) - h(x) \leq \frac{1}{2} \lVert z - y
  \rVert^{2} - \frac{1}{2} \lVert z - \overline{x} \rVert^{2} - \frac{1}{2}
  \lVert y - \overline{x} \rVert^{2}, \ \forall  y \in S_{h(\overline{x})} (h)
  \notag \\
  & \Longleftrightarrow ~ h(\overline{x}) - h(x) \leq \langle \overline{x} - z,
  x - \overline{x} \rangle, \ \forall  y \in S_{h(\overline{x})} (h). \notag
 \end{align}
Therefore, $h \in \Phi(S_{h(\overline{x})} (h))$.
\end{proof}

Another consequence is the following sufficient condition for $\inf$-compactness under an $L$-Lipschitz assumption, which revisits \cite[Corollary 1]{VI1}.

\begin{corollary}
 Let $h: \mathbb{R}^{n} \rightarrow \mathbb{R}$ be an $L$-Lipschitz and strong\-ly quasiconvex func\-tion. Then $h$ is $\inf$-compact on $\mathbb{R}^{n}$.
\end{corollary}

\begin{proof}
 If $h$ is strongly quasiconvex, then $h$ is neatly quasiconvex, and since $h$ is
 $L$-Lipschitz, $\partial^{\leq} h (x) \neq \emptyset$ for all $x \in \mathbb{R}^{n}$
 by Lemma \ref{nonempty:subd}$(b)$. Now, by Proposition \ref{strong:qcx}$(b)$, it
 follows that $h$ is positively quasiconvex on $\mathbb{R}^n$. Finally, $h$ is
 $\inf$-compact on $\mathbb{R}^{n}$ by \cite[Corollary 3.6]{CFZ}.
\end{proof}

We finish this section with the following observation.

\begin{remark}\label{rem3.6}
 There are (classes of) prox-convex functions which are neither convex nor strongly
 quasiconvex. Indeed, for all $n \in \mathbb{N}$, we take $K_{n} := [-n, +
 \infty[$ and the conti\-nuous quasiconvex functions $h_{n}: K_{n} \rightarrow
 \mathbb{R}$ given by $h_{n} (x) = x^{3}$. Clearly, $h_{n}$ is neither convex nor
 strongly quasiconvex on $K_{n}$ hence also not strongly G-subdifferentiable either.

 Take $n \in \mathbb{N}$. Then for all $z \in K_{n}$, ${\amin}_{K_{n}} h_{n} =
 \prox_{h_{n}} (z) = \{-n\}$, thus $S_{h_{n} (\overline{x})} (h_{n}) = \{
 \overline{x}\}$, i.e., $\partial^{\leq}_{K_{n}} h_{n} (\overline{x}) =
 \mathbb{R}^{n}$.  Therefore, $h_{n} \in \Phi(K_{n})$ for all $n \in \mathbb{N}$. Taking also into consideration Corollary \ref{cor3.1} one can conclude that the classes of strongly quasiconvex and prox-convex functions intersect without being included in one another.
 \end{remark}
\begin{remark}\label{rem3.7}
All the prox-convex functions we have identified so far are semi\-strict\-ly quasiconvex, too, while there are semi\-strict\-ly quasiconvex functions that are not prox-convex (for instance $h:\mathbb{R}\to \mathbb{R}$ defined by $h(x) = 1$ if $x=0$ and $h(x) = 0$ if $x \neq 0$), hence the connection between the classes of prox-convex and semistrictly quasiconvex functions remains an open problem.
\end{remark}

For a further study on strong quasiconvexity, positive quasiconvexity and $\inf$-compactness we refer to \cite{CFZ,VI1,VI2}.

\section{Proximal point type algorithms for nonconvex problems}\label{sec:4}

In this section we show that the proximal point type algorithm remains convergent
when the function to be minimized is proper, lower semicontinuous and prox-convex (on a given closed convex set), but
not necessarily convex. Although the algorithm considered below is the simplest and most basic version available and some of the advances achieved in
the convex case, such as accelerations and additional flexibility by employing additional parameters, are at the moment still open in the prox-convex setting,
our investigations show that the proximal point type methods can be
successfully extended towards other classes of nonconvex optimization problems for
which they could not be employed so far due to lack of a theoretical fundament.



\begin{theorem}\label{usual:prox}
Let $K$ be a closed and convex set in $\mathbb{R}^{n}$ and $h: \mathbb{R}^{n} \rightarrow \overline{\mathbb{R}}$ be a proper, lower semicontinuous and prox-convex on $K$
function such that ${\amin}_{K} h\neq \emptyset$ and $K \cap \dom \,h \neq \emptyset$. Then for any
 $k \in \mathbb{N}$, we set
 \begin{equation}\label{step:min}
  x^{k+1} = {\prox}_{h} (K, x^{k})
 \end{equation}
Then $\{x^{k}\}_k$ is a minimizing sequence of $h$ over $K$, i.e., $h(x^{k}) \rightarrow \min_{x \in K} h(x)$ when $k\to +\infty$.
\end{theorem}

\begin{proof}
Since $h$ is prox-convex on $K$, denote its prox-convex value by $\alpha > 0$ and for all $k \in \mathbb{N}$ one has
 \begin{align}
  & x^{k+1} = {\prox}_{h} (K, x^{k}) ~ \Longrightarrow ~ x^{k}
  - x^{k+1} \in \partial \left(\frac 1{\alpha} h + \delta_K\right)(x^{k+1}) \notag \\
  & \Longleftrightarrow \alpha \langle x^{k} - x^{k+1}, x - x^{k+1} \rangle \leq
   h(x) - h(x^{k+1}), \ \forall  x \in K.
  \label{alpha:subd}
 \end{align}
 Take $x = x^{k} \in K$, and since $\alpha > 0$, we have
  $$ 0 \leq \langle x^{k} - x^{k+1}, x^{k} - x^{k+1} \rangle \leq \frac{1}{\alpha} (h(x^{k})
  - h(x^{k+1})) \Longrightarrow ~ h(x^{k+1}) \leq h(x^{k}), ~ \forall  k \in \mathbb{N}.$$

 On the other hand, take $\overline{x} \in {\amin}_{K} h$. Then,
 for any $k \in \mathbb{N}$, by taking $x = \overline{x}$ in equation
 \eqref{alpha:subd}, we have
 \begin{align}
  \lVert x^{k+1} - \overline{x} \rVert^{2} & = \lVert x^{k+1} - x^{k} + x^{k}
  - \overline{x} \rVert^{2} \notag \\
  & = \lVert x^{k+1} - x^{k} \rVert^{2} + \lVert x^{k} - \overline{x} \rVert^{2}
  + 2\langle x^{k+1} - x^{k}, x^{k} - \overline{x} \rangle \notag \\
  & = -\lVert x^{k+1} - x^{k} \rVert^{2} + \lVert x^{k} - \overline{x} \rVert^{2}
  + 2\langle x^{k+1} - x^{k}, x^{k+1} - \overline{x} \rangle \notag \\
  & \leq \lVert x^{k} - \overline{x} \rVert^{2}
  +  \frac{2}{\alpha} (h(\overline{x}) - h(x^{k+1})) \leq \lVert x^{k} - \overline{x} \rVert^{2}, \label{useful:eq}
 \end{align}
 where we used that $h(\overline{x}) \leq h(x^{k+1})$. Thus, $\{x^{k} - \overline{x}\}_k$ is bounded.
 Then, and passing to a subsequence if needed, $x^{k} \rightarrow \overline{x}$ when
 $k \rightarrow + \infty$. Finally, since $h$ is lower semicontinuous and $K$ is closed, we have
 $\liminf_{k \rightarrow + \infty} h(x^{k}) = \min_{x \in K} h(x)$.
\end{proof}

\begin{remark}
 From \eqref{useful:eq} one can deduce straightforwardly that the known
 ${\mathcal{O}}(1/n)$ rate of convergence of the proximal point algorithm holds
 in the prox-convex case, too.
\end{remark}

\begin{remark}
 Although the function to be minimized in Theorem \ref{usual:prox} by means of the
 proximal point algorithm is assumed to be prox-convex, its prox-convex value
 $\alpha > 0$ needs not be known, even if it plays a role in the proof.
\end{remark}

\begin{remark}
 One can modify the proximal point algorithm by replacing in \eqref{step:min} the
 proximal step by ${\prox}_{h} (S_{h(x^{k})} (h), x^{k})$ without affecting the
 convergence of the generated sequence. Note also that taking $K=\mathbb{R}^{n}$
 in Theorem \ref{usual:prox} one ob\-tains the classical proximal point algorithm
 adapted for prox-convex functions and not for a restriction of such a function to
 a given closed convex set $K\subseteq \mathbb{R}^{n}$.
\end{remark}

\begin{example}
 Let $K=[0, 2] \times \mathbb{R}$ and consider the function $h: K \rightarrow
 \mathbb{R}$ given by $ h(x_{1}, x_{2}) = x^{2}_{2} - x^{2}_{1} - x_{1}$.
Observe that $h$ is continuous, strongly quasiconvex in the second argument,
 convex and strongly quasiconvex in the first argument, hence $h$ is strongly
 quasiconvex without being convex on $K$. Furthermore, by Example \ref{main:example}
$h$ is prox-convex on $K$.
The global minimum of $h$ over $K$ is $(2, 0)^\top$ and it can be found by applying
Theorem \ref{usual:prox}, i.e. via the proximal point algorithm, although the
function $h$ is not convex. First one determines the proximity operator
$${\prox}_h(K, (z_1, z_2)^\top )= \left(\left\{
\begin{array}{cc}
  0, & \mbox{if } z_1\leq -2 \\
  2, & \mbox{if } z_1> -2
\end{array},\right. \frac{z_2}{3}\right)^\top,\ z_1, z_2\in \mathbb{R}.$$
Taking into consideration the way $K$ is defined, it follows that the proximal step in Theorem \ref{usual:prox} delivers
$x^{k+1} = (2, x^{k}_2/3)^\top$, where $x^k = (x^k_1, x^k_2)^\top$. Whatever feasible starting point $x^1\in K$ of the algorithm is chosen, it delivers the global minimum of $h$ over $K$ because $x^k_1=2$ and $x^k_2=x^1_2/(3^{k-1})$ for all $k\in \mathbb{N}$.
\end{example}



\section{Conclusions and future work}

We contribute to the discussion on the convergence of proximal point algorithms
beyond convexity by introducing a new gene\-ra\-lized convexity notion called
\textit{prox-convexity}. We identify some classes of quasiconvex, weakly convex and DC functions (and
not only) that satisfy the new definition and different useful pro\-per\-ties of
these functions are proven. Then we show that the classical proximal point
algorithm remains con\-vergent when the convexity of the proper lower semicontinuous function to be minimized
is relaxed to prox-convexity (on a certain subset of the domain of the function).

In a future work, we aim to uncover more properties and develop calculus rules
for prox-convex functions as well as to extend our investigation to nonconvex
equilibrium problems and nonconvex mixed variational inequalities, to Hilbert
spaces and to splitting methods, also employing Bregman distances instead of the
classical one where possible.


\section{Declarations}

\subsection{Funding}

This research was partially supported by FWF (Austrian Science Fund), project
M-2045, and by DFG (German Research Foundation), project GR 3367/4-1 (S.-M. Grad)
and Conicyt--Chile under project Fondecyt Iniciaci\'on 11180320 (F. Lara).

\subsection{Conflicts of interest/Competing interests}

There are no conflicts of interest or competing interests related to this manuscript.

\subsection{Availability of data and material}

Not applicable.


\subsection{Code availability}

Not applicable.

\subsection{Authors' contributions}

Both authors contributed equally to the study conception and design.

\end{document}